\documentclass[pdflatex,sn-mathphys-num]{sn-jnl}
\usepackage{iftex}
\usepackage[T1]{fontenc}
\ifPDFTeX
\usepackage[utf8]{inputenc}
\fi
\usepackage{lmodern}
\usepackage{amsmath,amssymb,amsthm,mathtools}
\usepackage{mathrsfs}
\DeclareFontFamily{U}{rsfs}{\skewchar\font127 }
\DeclareFontShape{U}{rsfs}{m}{n}{%
  <-6.5> rsfs5
  <6.5-8> rsfs7
  <8-> rsfs10
}{}
\usepackage{microtype,booktabs,xcolor}
\usepackage[nameinlink,noabbrev]{cleveref}
\hypersetup{hidelinks,hypertexnames=false}

\newtheorem{theorem}{Theorem}[section]
\newtheorem{lemma}[theorem]{Lemma}

\theoremstyle{remark}

\numberwithin{equation}{section}
\newcommand{\Z}{\mathbb Z}
\newcommand{\PP}{\mathbb P}
\newcommand{\EE}{\mathbb E}
\newcommand{\Dir}{\mathcal D}
\newcommand{\T}{\mathcal T}
\newcommand{\Cp}{C^{+}}
\newcommand{\thetab}{\theta_{\mathrm b}}
\newcommand{\rights}{\rightsquigarrow}
\makeatletter
\let\PTRForiginalheadings\ps@headings
\def\ps@headings{%
  \PTRForiginalheadings
  \let\@oddhead\@empty
  \let\@evenhead\@empty}
\makeatother

\begin{document}

\title[Two-neighbor lattice percolation percolates in all dimensions]{Two-neighbor lattice percolation percolates in all dimensions}

\author[1]{\fnm{Jing} \sur{Liu}}\email{jingliu0813@126.com}
\author[1]{\fnm{Kainan} \sur{Xiang}}\email{kainan.xiang@xtu.edu.cn}
\author[2]{\fnm{Lang} \sur{Zou}}\email{langzou@xtu.edu.cn}

\affil[1]{\orgdiv{Hunan Key Laboratory for Computation and Simulation in Science and Engineering \& Key Laboratory of Intelligent Computing and Information Processing of Ministry of Education \& School of Mathematics and Computational Science}, \orgname{Xiangtan University}, \orgaddress{\city{Xiangtan City} \postcode{411105}, \state{Hunan Province}, \country{P. R. China}}}
\affil[2]{\orgdiv{School of Artificial Intelligence and Statistics}, \orgname{Xiangtan University}, \orgaddress{\city{Xiangtan City} \postcode{411105}, \state{Hunan Province}, \country{P. R. China}}}

\abstract{
It is trivial that $1$-dimensional $2$-neighbor lattice percolation percolates. Coupier, Henry, Jahnel and K\"oppl proved
the $2$-neighbor lattice percolation percolates in dimension 2; and then for sufficiently large dimensions, B\"aumler, Jahnel, K\"oppl, Lodewijks, Reeves and T\'obi\'as
showed this is still true. In this paper, we verify that the $2$-neighbor lattice percolation percolates in all dimensions $d\ge 3.$
}

\maketitle

\section{\bf The model and the result}\label{sec:model}
\setcounter{equation}{0}
\noindent
Let $d\in\mathbb{N}=\{1,\, 2,\,\cdots\}$ and $e_1,\,\cdots,\, e_d$ be the standard coordinate
vectors of $\Z^d$. Write
$$\Dir_d=\{e_1,-e_1,\, e_2,-e_2,\, \cdots,\, e_d,-e_d\}.$$
Given any $1\le k\leq 2d.$ Recall the following $k$-neighbor lattice percolation (random graph) was introduced by Jahnel, K\"oppl, Lodewijks and T\'obi\'as (2025) \cite{Ja-Ko2025}:
\begin{eqnarray*}
&&\text{At any vertex $x\in\mathbb{Z}^d,$ independently choose a uniformly distributed
$k$-element}\\
&&\text{subset $S_x$ of $\Dir_d$. Put a directed arrow from $x$
to $x+v$ for each $v\in S_x$.}
\end{eqnarray*}
The $k$-neighbor lattice percolation is a discrete counterpart of the continuum $k$-nearest-neighbor model. The latter was introduced by H\"{a}ggstr\"{o}m and Meester (1996) \cite{Ha-Me1996}, and then
extended by Balister and Bollob\'{a}s (2013) \cite{Ba-Bo2013}, and has attracted significant attention. A feature of the $k$-neighbor lattice percolation is that it can be viewed as an $1$-dependent Bernoulli bond directed percolation
and exhibits a negative correlation property. For a relevant $k$-out percolation model, refer to Amir, Heydenreich and Hirsch (2025) \cite{Am-He2025}.

Let $\PP_{k,d}^{+}$ be the law of
this random graph. Use $x\rights y$ to denote that there is a finite directed path
from $x$ to $y$ (including path of length zero). Set
\[
 \Cp(x)=\left\{y\in\mathbb{Z}^d:\ x\rights y\right\}\ \mbox{for}\ x\in\mathbb{Z}^d,\ \mbox{and}\ \theta^{+}(k,d)=\PP_{k,d}^{+}\left(\left|\Cp(\mathbf{0})\right|=\infty\right),
\]
where $\mathbf{0}$ is the origin. Clearly, the existence of an infinite outgoing cluster $\Cp(x)$ is equivalent to that of an infinite self-avoiding directed path from $x$. Jahnel, K\"oppl,
Lodewijks and T\'obi\'as \cite{Ja-Ko2025} initiated study on the $k$-neighbor lattice percolation and conjectured that
$$ \theta^{+}(2,d)>0,\ d\in\mathbb{N}.$$
See \cite[Conjecture 3.1]{Ja-Ko2025}.

Trivially $\theta^{+}(2,1)>0$. From Coupier, Henry, Jahnel and K\"oppl (2026) \cite[Theorem 2.3]{Co-He2026},
$\theta^{+}(2,2)>0;$ and the related proof is highly nontrivial. The local connectivity comparison of B\"aumler, Jahnel, K\"oppl, Lodewijks, Reeves and T\'obi\'as (2026) \cite{Bau-Ja2026} yields that the
2-neighbor percolation percolates in all sufficiently large dimensions \cite[Corollary 3.7]{Bau-Ja2026}.

To state our main result, write $\thetab(p,d)$ for the probability of an infinite open cluster at the origin in ordinary undirected Bernoulli bond percolation on $\Z^d$ with each edge being open independently with probability $p\in [0,1]$.
Note the critical probability of the Bernoulli bond percolation on $\Z^2$ is $1/2$ (Harris (1960) \cite{Ha1960}, Kesten (1980) \cite{Ke1980}). Then we have

\begin{theorem}\label{thm:main}
For any integer $d\geq 3$,
\begin{equation}
\label{eq:uniformbound}
                 \theta^{+}(2,d)\geq\thetab(11/20,2)>0.
\end{equation}
\end{theorem}

By Theorem \ref{thm:main} and the previous results mentioned above, the $2$-neighbor lattice percolation percolates in all dimensions, namely \cite[Conjecture 3.1]{Ja-Ko2025} holds.

Use an independent uniform permutation of the $2d$ direction labels at each vertex $x\in\mathbb{Z}^d$, and take the first $k$ labels. This
couples the $2$-neighbor lattice percolation as a subgraph of every $k$-neighbor one with $k\geq 2$, proving $\theta^{+}(k,d)>0\ (k\geq 2)$ by Theorem \ref{thm:main}.
From \cite[Remark 2.1(i)]{Ja-Ko2025}, $\theta^{+}(1,d)=0$ for any $d\in\mathbb{N}.$ Hence, for the $k$-neighbor lattice percolation on $\mathbb{Z}^d$ with $1\leq k\leq 2d$,
$$  \theta^{+}(k,d)>0 \, \Longleftrightarrow\, k\geq 2.$$

For more information and questions (conjectures) on the $k$-neighbor lattice percolations, see \cite{Ja-Ko2025,Co-He2026,Bau-Ja2026,Am-He2025}.
\vskip 2mm

Our contribution is a fiber exploration and an exact avoidance transformation
that establish the bound \eqref{eq:uniformbound} for any $d\geq 3$, including a constant independent of the dimension.
The final comparison \eqref{eq:comparison} uses \cite[Theorem 2.1 and Lemma 3.5]{Bau-Ja2026}. For completeness, we present \cite[Theorem 2.1 and Lemma 3.5]{Bau-Ja2026} in \ref{app-1}
respectively by Theorem \ref{thm:local-global} and Lemma \ref{lem:one-arm}.

The mechanism is to collect all horizontal exits that can be reached
along a vertical fiber from its first entrance. The resulting
direction set can have fewer than two elements, so it is not a
lower-dimensional two-neighbor model. Nevertheless, its probabilities
of hitting specified direction sets admit exact recursive estimates.
After repeated projection, those estimates exceed the corresponding
probabilities in a supercritical planar independent-arrow model.
It is these hitting probabilities, rather than a fixed outgoing
degree, that permit the final connectivity comparison.

\section{\bf Proof of Theorem \ref{thm:main}}
\setcounter{equation}{0}
\noindent

\subsection{\bf Independent local environments}\label{sec:local}
It is useful to consider a larger class of models. In
dimension $n$, let $\mu$ be a probability law on
$$\mathcal{P}(\Dir_n)=\{A:\ A\subseteq \Dir_n\}.$$
Independently at each $x\in\Z^n$, sample $S_x$
with law $\mu$ and put arrows $x\to x+v$ for $v\in S_x$.
Write $\mathbb{P}_{\mu}^{+}$ for the law of this independent-arrow model, and let
$$\theta(\mu,n):=\theta_{\mathrm{arrow}}(\mu,n)=\PP_{\mu}^{+}\left(\left|\Cp_{\mu}(\mathbf{0})\right|=\infty\right).$$
Here $\Cp_{\mu}(\mathbf{0})$ is the outgoing cluster at the origin $\mathbf{0}.$
Call $\mu$ exchangeable if it is invariant under every permutation
of the $2n$ direction labels. For a specified independent local environment, \textit{for conciseness}, write $\PP$ and
$\EE$ respectively for probability and expectation.

Assume $\mu$ is exchangeable and put $q=2n$. For a random set $S$ of law $\mu$, its avoidance probabilities are
\begin{eqnarray}\label{eq:avoidance}
 f_m=\mu(S\cap A=\varnothing),\ A\subseteq\Dir_n,\ |A|=m,\ 0\leq m\leq q.
\end{eqnarray}
Exchangeability makes the probability independent of the particular
set $A$. Thus
$$f_0=1\ \mbox{and}\ 1\geq f_1\geq\cdots\geq f_q\geq 0.$$
For the uniform two-neighbor law in degree $q$, this sequence is
\begin{equation}\label{eq:useq}
 u^{q}_m=\frac{\binom{q-m}{2}}{\binom q2}
       =\frac{(q-m)(q-m-1)}{q(q-1)},\ 0\leq m\leq q.
\end{equation}
The binomial coefficient $\binom{q-m}{2}$ is zero for any $m\in\{q-1,q\}.$ Note here we use superscripts on $u^q=\left(u^q_m\right)_{0\le m\le q}$ to specify the degree, not
to denote powers.

\subsection{\bf Exploring a coordinate fiber}\label{sec:fiber}

Suppose $n\geq3$ and $0<f_1<1$. Write vertices of $\Z^n$ as
$(y,t)$ with $y\in\Z^{n-1}$ and $t\in\Z$, and let
$\pi(y,t)=y$. The fiber over $y$ is
$$F_y=\{(y,t):\ t\in\Z\}.$$
Directions $\pm e_n$ are vertical, and the other $q-2$ directions are horizontal.

For an entrance height $h$, let

\[
 R_y(h)=\{t\in\Z:\ (y,h)\rights (y,t)\text{ using only vertical arrows}\}.
\]
This is almost surely a finite interval. To reach $h+j\ (j\ge 1)$ vertically,
each of the $j$ successive arrows in the positive direction must
be present. These arrows have different tails and hence are
independent with common probability $1-f_1<1$. Consequently
$$\PP\left(h+j\in R_y(h)\right)=(1-f_1)^j\ \mbox{for $j\geq1$},$$
and the analogous
identity holds on the negative ray. Particularly, neither ray
can be reached to infinity; and further these finiteness statements hold
simultaneously for all $(y,h)\in\mathbb{Z}^{n-1}\times \mathbb{Z}$.

Define the set of horizontal exit directions accessible from entrance $(y,h)$ by
\begin{equation}\label{eq:W}
 W_y(h)=\left\{v\in\Dir_{n-1}:\ v\in S_{(y,t)}
                    \text{ for some }t\in R_y(h)\right\}.
\end{equation}
Here $\Dir_{n-1}$ is viewed naturally as a subset of $\Dir_n$ by letting the $n$th coordinate be zero. Let $\widehat\mu$ be the law of $W_0(0)$. The environments of
different fibers are independent and stationary under shifts
of the height coordinate. So $W_y(h)$ has this same law
for every deterministic $y,h$. Permuting the horizontal direction
labels shows that $\widehat\mu$ is exchangeable.

\begin{lemma}[First-entrance coupling]\label{lem:projection}
There is a coupling of $\mathbb{P}_{\mu}^{+}$ on $\Z^n$ and $\mathbb{P}_{\widehat\mu}^{+}$ on $\Z^{n-1}$ such that their
outgoing clusters at the origin satisfy
\begin{equation}\label{eq:clusterprojection}
             \Cp_{\widehat\mu}(\mathbf{0})\subseteq\pi\,\Cp_{\mu}(\mathbf{0}).
\end{equation}
Particularly
\begin{equation}\label{eq:thetaproj}
                    \theta(\mu,n)\geq\theta\left(\widehat\mu,n-1\right).
\end{equation}
\end{lemma}

\begin{proof}
Write $\mathcal E_y=\left(S_{(y,t)}\right)_{t\in\Z}$ for the entire environment
of the fiber over $y\in\Z^{n-1}$. These fiber environments are independent,
identically distributed, and invariant in law under integer shifts
of the height coordinate.

Run a first-in, first-out exploration on the projected lattice $\mathbb{Z}^{n-1}.$
Initially $Y_0=\mathbf{0}\in\Z^{n-1}$ is discovered,
and its entrance
height is $h_{Y_0}=0$, expose $\mathcal E_{Y_0}$, compute
$R_{Y_0}(h_{Y_0})$ and $W_{Y_0}(h_{Y_0})$, and use $W_{Y_0}(h_{Y_0})$ as its outgoing
direction set in the projected exploration. Let
$$\mathcal{P}_0=\{Y_0\},\ \mathcal{Q}_0=\left\{Y_0+v:\ v\in W_{Y_0}(h_{Y_0})\right\}\ (\mbox{the newly reached neighbor set of $Y_0$}).$$
For any $Y_0+v\in\mathcal{Q}_0$ (if exists), choose the smallest eligible height
$t\in R_{Y_0}(h_{Y_0})$ with $v\in S_{(Y_0,t)}$, and set
$h_{Y_0+v}=t$.  Let $\mathcal H_0$ be
$$\text{the $\sigma$-algebra generated by $\mathcal{P}_0$, $\{\mathcal E_y\}_{y\in\mathcal{P}_0},$ $\mathcal{Q}_0$ and $\{h_y\}_{y\in\mathcal{Q}_0}$.}$$

When $\mathcal{Q}_0$ is empty, then leave the exploration state unchanged at any subsequent update, and let
$$Y_1=\mathbf{0},\ h_{Y_1}=0,\ \mathcal{P}_r=\mathcal{P}_0,\, r\ge 1.$$
Otherwise, choose the first $Y_1\in\mathcal{Q}_0$ according to the lexicographic order, and let
$$\mathcal{P}_1=\{Y_0,Y_1\},\ \mathcal{Q}_1=(\mathcal{Q}_0\setminus\{Y_1\})\cup \left(\left(Y_1+W_{Y_1}(h_{Y_1})\right)\setminus\mathcal{Q}_0\right); $$
and for any $Y_1+v\in Y_1+W_{Y_1}(h_{Y_1})$ (if exists), choose the smallest eligible height
$t\in R_{Y_1}(h_{Y_1})$ with $v\in S_{(Y_1,t)}$, and set
$h_{Y_1+v}=t$. Note for any $y\in\mathcal{Q}_1$ (if exists), $h_y$ is defined. And $\left(Y_1+W_{Y_1}(h_{Y_1})\right)\setminus\mathcal{Q}_0$ is
the newly reached neighbor set of $Y_1$.

\vskip 2mm

Note if we just have $j\ge 1$ updates and $\mathcal{Q}_{j-1}\not=\varnothing,$ then $Y_j\in\mathcal{Q}_{j-1}$ is the first element of $\mathcal{Q}_{j-1}$ in the lexicographic order,
$$\mathcal{P}_j=\{Y_0,Y_1,\cdots,Y_j\},\ \mathcal{Q}_j=(\mathcal{Q}_{j-1}\setminus\{Y_j\})\cup \left(\left(Y_j+W_{Y_j}(h_{Y_j})\right)\setminus\mathcal{Q}_{j-1}\right); $$
and for any $Y_j+v\in Y_j+W_{Y_j}(h_{Y_j})$ (if exists), choose the smallest eligible height
$t\in R_{Y_j}(h_{Y_j})$ with $v\in S_{(Y_j,t)}$, and set
$h_{Y_j+v}=t$. Note for any $y\in\mathcal{Q}_j$ (if exists), $h_y$ is defined. Let $\mathcal H_j$ be
$$\text{the $\sigma$-algebra generated by $\mathcal{P}_j$, $\{\mathcal E_y\}_{y\in\mathcal{P}_j},$ $\left\{\mathcal{Q}_r\right\}_{r\leq j}$ and $\{h_y\}_{y\in\cup_{r\le j}\mathcal{Q}_r}$.}$$

When $\mathcal{Q}_j$ is empty, then leave the exploration state unchanged at
any subsequent update, and let
$$ Y_{j+1}=\mathbf{0},\ h_{Y_{j+1}}=0,\ \mathcal{P}_r=\mathcal{P}_j,\, r\ge j+1.$$
Otherwise, choose the first $Y_{j+1}\in\mathcal{Q}_j$ according to the lexicographic order, and let
$$\mathcal{P}_{j+1}=\{Y_0,Y_1,\cdots,Y_{j+1}\},\ \mathcal{Q}_{j+1}=(\mathcal{Q}_j\setminus\{Y_j\})\cup \left(\left(Y_{j+1}+W_{Y_{j+1}}(h_{Y_{j+1}})\right)\setminus\mathcal{Q}_j\right); $$
and for any $Y_{j+1}+v\in Y_{j+1}+W_{Y_{j+1}}(h_{Y_{j+1}})$ (if exists), choose the smallest eligible height
$t\in R_{Y_{j+1}}(h_{Y_{j+1}})$ with $v\in S_{(Y_{j+1},t)}$, and set
$h_{Y_{j+1}+v}=t$. Note for any $y\in\mathcal{Q}_{j+1}$ (if exists), $h_y$ is defined; and $\left(Y_{j+1}+W_{Y_{j+1}}(h_{Y_{j+1}})\right)\setminus\mathcal{Q}_j$ is
the newly reached neighbor set of $Y_{j+1};$ and
$$\text{both $Y_{j+1}$ and $h_{Y_{j+1}}$ are
$\mathcal H_j$-measurable, and $Y_{j+1}\notin\mathcal P_j$ when $\mathcal{Q}_j\not=\varnothing.$}
$$

Conditionally on $\mathcal H_j$, the environments $\mathcal{E}_y$ at the unprocessed
indices $y$ still have their original independent product law. 
The next index $Y_{j+1}$ is chosen using only the recorded history $\mathcal{H}_j;$ exposing its fresh
environment and then updating $\mathcal{Q}_{j+1}$ reveals no information about
the remaining environments. In particular, independence and height
stationarity imply, for any bounded measurable function $\varphi$
on horizontal direction sets,
\[
 \EE\!\left[
 \varphi\bigl(W_{Y_{j+1}}(h_{Y_{j+1}})\bigr)
 \,\middle|\,\mathcal H_j\right]
             =\int\varphi\,\mathrm d\widehat\mu\ \mbox{on}\ \mathcal{Q}_j\not=\varnothing.
\]

To complete an independent field on the entire projected lattice,
take auxiliary independent sets $V_y$ with law
$\widehat\mu$, independent of the original environments $\mathcal{E}_y,$ and never
consult them during the exploration. For any $j\ge 0,$ define the following random field $X^j=\left(X^j_y\right)_{y\in\Z^{n-1}}:$
\[
 X^j_y=\begin{cases}
          W_y(h_y),&y\in\mathcal P_j,\\
          V_y,&y\notin\mathcal P_j.
        \end{cases}
\]
Given $\mathcal H_j$, all auxiliary sets are still
independent with their original law. Conditionally on $\mathcal H_j,$
on $\mathcal{Q}_j\not=\varnothing$, the new
set $W_{Y_{j+1}}(h_{Y_{j+1}})$ has law $\widehat\mu$ and is
independent of those auxiliary sets.
The processed coordinates are fixed by this history $\mathcal H_j$. Consequently,
for any bounded function $\Phi$ depending on finitely many field
coordinates,
\[
       \EE\left[\left.\Phi\left(X^{j+1}\right)\,\right\vert\,\mathcal H_j\right]
                 =\EE\left[\left.\Phi(X^j)\,\right\vert\,\mathcal H_j\right]\ \mbox{on}\ \mathcal{Q}_j\not=\varnothing.
\]
The equality also holds on $\mathcal{Q}_j=\varnothing$ because the field is then
unchanged. Taking expectations and inducting on $j$ proves that
each $X^j$ has product law with one-site law $\widehat\mu$.
\vskip 2mm

Every coordinate changes at most once. Thus $X^j$ converges
coordinatewise almost surely to a random field $X^\infty$ that uses
$W_{Y_j}(h_{Y_j})$ at every eventually processed vertex $Y_j$ and $V_y$ at all
others. A function depending on finitely many coordinates is
eventually unchanged along this sequence. Bounded convergence theorem
therefore shows that $X^\infty$ also has the required product law.

Every discovered projected vertex $Y_j$ has an accessible representative $(Y_j,h_{Y_j})$
in the original graph. This holds initially at $(\mathbf{0},0)$. If
$(y,h_y)$ is accessible and $y+v$ is discovered from $y$, a
vertical path reaches $(y,t)$ and the selected horizontal arrow
reaches $(y+v,t)=(y+v,h_{y+v})$. Induction proves accessibility
of all assigned representatives. Each update adds at most $q-2$
vertices, so each $\mathcal{Q}_j$ is finite. The first-in,
first-out rule therefore ensures that every discovered vertex is
eventually processed.

Let $D=\cup_{j\geq 0}\mathcal{P}_j.$ The first-discovery tree uses arrows of $X^\infty$, so
$D$ is contained in $\Cp_{X^\infty}(\mathbf{0}),$ the outgoing cluster at $\mathbf{0}\in\Z^{n-1}$ for $X^\infty.$
Conversely, each vertex in $D$ is
eventually processed, and every neighbor specified by its final
direction set is then discovered. Thus $D$ contains the origin
and is closed under outgoing arrows of $X^\infty$, giving
$\Cp_{X^\infty}(\mathbf{0})\subseteq D$. Notice the first-discovery tree supplies a genuine
path to every discovered representative $(Y_j,h_{Y_j}),$ and distinct projected vertices $Y_j$ have
distinct representatives $(Y_j,h_{Y_j}).$ Then \eqref{eq:clusterprojection} and \eqref{eq:thetaproj} hold.
\end{proof}

\begin{lemma}[Exact avoidance transformation]\label{lem:transform}
The avoidance sequence of $\widehat\mu$ is $\T f$, where
for $0\leq m\leq q-2$,
\begin{equation}\label{eq:transform}
 (\T f)_m=G(f_m,f_{m+1},f_{m+2}),\
 G(a,b,c)=\frac{c(1-a)^2+(2-a)b^2}{(1-a+b)^2}.
\end{equation}
\end{lemma}

\begin{proof}
Fix a set $A$ of $m$ horizontal directions. On one vertical ray,
let $t$ be the probability that all horizontal exits from the
vertically reachable part of that ray avoid $A$, conditional on
having just entered its first vertex. A backward vertical arrow
only returns to the preceding reachable vertex, and cannot create
a new vertex on this ray.

At the first vertex, either $A$ is avoided and the forward
vertical direction is absent, an event of probability $f_{m+1}$,
or $A$ is avoided and that direction is present, an event of
probability $f_m-f_{m+1}$. In the second case the remaining
ray has a fresh independent environment and contributes another
factor $t$. Hence
\begin{equation}\label{eq:t}
             t=f_{m+1}+(f_m-f_{m+1})t,
             \ t=\frac{f_{m+1}}{1-f_m+f_{m+1}}.
\end{equation}
The denominator is positive. For $m=0$, it equals $f_1>0$;
for $m\geq1$, $f_m\leq f_1<1$.
In both cases $f_m-f_{m+1}<1$, so this renewal equation has
a unique solution. The almost sure finiteness of the reachable
ray justifies the decomposition without an additional boundary
condition at infinity.

At the entrance vertex, avoidance of $A$ together with absence
of both vertical directions has probability $f_{m+2}$.
Avoidance together with precisely one specified vertical direction
has probability $f_{m+1}-f_{m+2}$. Avoidance together with
presence of both has probability
$$f_m-2(f_{m+1}-f_{m+2})-f_{m+2}=f_m-2f_{m+1}+f_{m+2}.$$
The two rays beyond the entrance
are independent. Summing the three possibilities gives
\begin{align}
 (\T f)_m
   &=f_{m+2}+2(f_{m+1}-f_{m+2})t
          +(f_m-2f_{m+1}+f_{m+2})t^2\notag\\
   &=(1-t)^2f_{m+2}+2t(1-t)f_{m+1}+t^2f_m.\label{eq:convex}
\end{align}
Substitute \eqref{eq:t} to obtain \eqref{eq:transform}.
\end{proof}


\subsection{\bf Bounds for iterated projection}\label{sec:algebra}

\begin{lemma}[Monotonicity and locality of the transformation]\label{lem:isotone}
The function $G(a,b,c)$ is nondecreasing in each variable on
\[
                0\leq c\leq b\leq a<1;
\]
and satisfies $0\leq G(a,b,c)\leq a$.

Let $q\geq2$ and
$$\mathscr{F}=\left\{f=(f_0,\cdots,f_q):\ f_0=1,\ 0<f_1<1,\ f_1\geq\cdots\geq f_q\geq 0\right\}.$$
Given any $f\in\mathscr{F}.$ Define $\T$ algebraically by $(\T f)_m=G(f_m,f_{m+1},f_{m+2})$ for $0\leq m\leq q-2.$
Then $\T f$ is well defined, and
$$(\T f)_0=1,\  1>(\T f)_1\ge \cdots\ge (\T f)_{q-2}\ge 0\ \mbox{when}\ q>2.$$

Let $g=(g_0,\cdots,g_q)\in\mathscr{F}.$ Suppose $f$ and $g$ admit $j\geq 0$ successive transformations, with
every intermediate input sequence in $\mathscr{F}.$ Then $(\T^r f)_0=1,\ 0\le r\le j;$ and
for any $m\geq 1$ and $m+2j\leq q$, $(\T^j f)_m$ depends only on
$f_m,\cdots,f_{m+2j}$. Moreover,
\[
 \bigl(f_\ell\leq g_\ell\text{ for }m\leq\ell\leq m+2j\bigr)
       \ \Longrightarrow\ (\T^j f)_m\leq(\T^j g)_m.
\]
\end{lemma}

\begin{proof}
Put $w=1-a>0$. Direct differentiation of
$G(a,b,c)=\frac{cw^2+(1+w)b^2}{(w+b)^2}$ gives
\begin{align*}
 \partial_cG&=\frac{w^2}{(w+b)^2},\\
 \partial_bG&=\frac{2w((1+w)b-cw)}{(w+b)^3},\\
 \partial_aG&=\frac{b\{b(2-b)+w(b-2c)\}}{(w+b)^3}.
\end{align*}
All are nonnegative: for the last derivative its brace is at least
$b(2-b-w)=b(1+a-b)\geq0$. The displayed domain is convex.
Integrating the gradient along the line segment between two
componentwise comparable triples in this domain proves the stated
nondecreasing property. Put $t=b/(1-a+b)\in[0,1]$ in
\eqref{eq:convex}; its algebraic identity expresses $G(a,b,c)$ as a
convex combination of $a,b,c$, proving $0\leq G(a,b,c)\leq a$.

For $m\geq1$, each triple defining $(\T f)_m$ belongs to this
domain, so its denominator is positive and its value is below one.
For $m=0$, direct substitution gives
$G(1,f_1,f_2)=1$, since $f_1>0$. For successive positive indices,
\[
 (f_m,f_{m+1},f_{m+2})
          \geq(f_{m+1},f_{m+2},f_{m+3})
\]
componentwise, wherever both triples exist. Monotonicity proves
that the transformed positive-index sequence is nonincreasing;
the zeroth component is larger than all of them.

Finally, a single update of component $m$ uses precisely the
index window $m,m+1,m+2$. For $j$ updates, its three arguments
use windows whose union is $m,\cdots,m+2j$. Induction on $j$
proves locality and the comparison implication: the assumed
initial inequalities give inequalities for all three arguments
at the last update, and monotonicity of $G$ completes the step.
This argument uses only positive indices. Clearly the zeroth component is always one during iterations.
\end{proof}

All iterated projections of the uniform two-neighbor law are
legitimate for the degrees considered below. Initially
$$u_0^q=1>u_1^q>u_2^q>\cdots>u^q_{q-2}>0.$$
If this property holds in degree $q\geq 6$ for $(g_0,g_1,\cdots,g_q),$ then
$$G(g_m,g_{m+1},g_{m+2})>0\ \mbox{for}\ 1\leq m\leq q-4,$$
because
$g_{m+2}>0$ and $1-g_m>0$. Also $G(g_m,g_{m+1},g_{m+2})\leq g_m<1$.
Thus, after iterations, in every resulting degree $r$ at least four, the obtained sequence $(f_0,f_1,\cdots,f_{r})$
satisfies
$$0<f_1<1=f_0,$$
as required by Lemma \ref{lem:projection}.

\begin{lemma}[One-step comparison away from the last two indices]\label{lem:onestep}
For each even $q\geq4$ and $0\leq m\leq q-2$,
\begin{equation}\label{eq:onestep}
                         (\T u^{q+2})_m\leq u^q_m.
\end{equation}
\end{lemma}

\begin{proof}
Put $r=q-m$, so $2\leq r\leq q$, and set
$D=(q+2)(q+1)$ and $L=q(q+3)-2r>0$.
For the triple from $u^{q+2}$,
\[
 a=\frac{(r+2)(r+1)}D,\ b=\frac{r(r+1)}D,\ c=\frac{r(r-1)}D.
\]
Using \eqref{eq:transform}, simplification gives
\begin{equation}\label{eq:polynomialidentity}
 \frac{r(r-1)}{q(q-1)}-G(a,b,c)
 =\frac{2r(q-r)P_q(r)}{Dq(q-1)L^2},
\end{equation}
where
\begin{align}
 P_q(r)={}&-(3q^2+5q+4)r^2
 +(2q^4+9q^3+15q^2+18q+4)r\notag\\
 &-(2q^4+13q^3+24q^2+9q).\label{eq:P}
\end{align}
Here is an intermediate expression making the identity checkable:
if $h=r(r+1)/L$, then
$G(a,b,c)=[r(r-1)+4rh+2h^2]/D$.
Clearing the denominators gives \eqref{eq:polynomialidentity}.

The polynomial $P_q$ is concave as a function of $r$, and
\begin{align*}
 P_q(2)&=2q^4+5q^3-6q^2+7q-8>0,\\
 P_q(q)&=q(2q^4+4q^3-3q^2-10q-5)>0.
\end{align*}
For $q\geq4$ these signs follow by grouping
$2q^4-6q^2>0$ in the first expression, and
$2q^4-3q^2>0$ and $4q^3-10q-5>0$ in the second.
Concavity yields $P_q(r)\geq\min\{P_q(2),P_q(q)\}>0$
on $[2,q]$. All factors on the right side of
\eqref{eq:polynomialidentity} are nonnegative. When $m=0$,
$q-r=0$ and both sides of \eqref{eq:onestep} equal one.
\end{proof}

The comparison cannot be extended to all indices. It fails at index
$q-1$, already when $q=4$: one projection from degree six to four gives
\[
                  (\T u^6)_3=\frac9{845}>0=u^4_3.
\]
By contrast, the single-step full-set avoidance probability satisfies
\[
 (\T u^{q+2})_q
   =G\!\left(\frac{2}{(q+2)(q+1)},0,0\right)=0=u^q_q.
\]
After further projections, however, the full-set avoidance probability
can be positive. Thus the last indices still require separate bounds
in an iterated comparison.

Let $Q\geq4$ be an even initial degree. For even $q$ with
$4\leq q\leq Q$, write
\begin{equation}\label{eq:iteratednotation}
                f^{Q\downarrow q}=\T^{(Q-q)/2}u^Q.
\end{equation}
This is the avoidance sequence obtained by repeated coordinate
fiber exploration, until the remaining lattice has degree $q$.

\begin{lemma}[Uniform bounds for the lower indices]\label{lem:iterated}
For every pair of such degrees,
\begin{equation}\label{eq:iteratedbound}
                 f^{Q\downarrow q}_m\leq u^q_m,
                 \ 0\leq m\leq q-2.
\end{equation}
\end{lemma}

\begin{proof}
The zeroth component is one on both sides, so assume $m\geq1$.
Use induction on $(Q-q)/2$. The zero-step case is equality.
For a positive number of steps, first apply Lemma \ref{lem:onestep}:
$\T u^Q\leq u^{Q-2}$ at indices at most $Q-4$.
There are $j=(Q-q)/2-1$ remaining transformations. Their final
component $m\leq q-2$ depends nondecreasingly only on the initial
indices at most
\[
                        m+2j\leq q-2+Q-q-2=Q-4.
\]
Thus Lemma \ref{lem:isotone} gives
$f^{Q\downarrow q}_m\leq f^{(Q-2)\downarrow q}_m$.
Apply the induction hypothesis.
\end{proof}

The restriction on the indices is substantive. Projecting a
two-neighbor model can yield only one horizontal exit direction,
so domination of a uniform two-neighbor law at every avoidance
index would be false. The proof uses only the indicated lower
indices, and bounds the remaining indices separately.

\subsection{\bf A uniform planar comparison and the proof}\label{sec:finish}

Fix $Q=2d\geq6$ and put $a_m=f^{Q\downarrow6}_m$ for
$0\leq m\leq6$. If $Q=6$, this is just $u^6$.
By Lemma \ref{lem:iterated},
\begin{equation}\label{eq:sixlower}
      a_1\leq\frac{2}{3},\ a_2\leq\frac{2}{5},\
      a_3\leq\frac{1}{5},\ a_4\leq\frac{1}{15}.
\end{equation}
Nonincrease gives $a_5\leq a_4\leq 1/15$.
For the last index we also have
\begin{equation}\label{eq:lastindex}
                                  a_6\leq\frac1{28}.
\end{equation}
{\it Indeed}, if $Q=6$ then $a_6=0$. If $Q\geq8$, then the last
projection from degree eight to six and
$$G(a,b,c)\leq a,\ 0\le c\le b\le a<1$$
give
$a_6\leq f^{Q\downarrow8}_6$; now Lemma \ref{lem:iterated} applies at index $6=8-2$, and
$u^8_6=1/28$.

Let $b_m=f^{Q\downarrow4}_m$ be the final planar avoidance
sequence. Apply $G$ to the bounds above, using Lemma \ref{lem:isotone}:
\begin{align}
 b_1&\leq G(2/3,2/5,1/5)=\frac{53}{121},\notag\\
 b_2&\leq G(2/5,1/5,1/15)=\frac{11}{80},\notag\\
 b_3&\leq G(1/5,1/15,1/15)=\frac{57}{845},\notag\\
 b_4&\leq G(1/15,1/15,1/28)=\frac{134}{3375}.\notag \label{eq:fourbounds}
\end{align}
The index $m=4$ is essential: $b_4$ is the probability of an empty
planar outgoing set, which can be positive after repeated projection.
Each $b_m$ $(1\le m\le 4)$ is strictly smaller than $(9/20)^m$. The exact differences are listed to make the
strict comparison transparent:
\begin{equation}\label{eq:gaps}
 \begin{aligned}
 \frac9{20}-\frac{53}{121}&=\frac{29}{2420},&
 \left(\frac9{20}\right)^2-\frac{11}{80}&=\frac{13}{200},\\
 \left(\frac9{20}\right)^3-\frac{57}{845}&=\frac{32001}{1352000},&
 \left(\frac9{20}\right)^4-\frac{134}{3375}&=\frac{5627}{4320000}.
 \end{aligned}
\end{equation}
\vskip 3mm

\begin{proof}[{\bf Proof of Theorem \ref{thm:main}}]
For $d\geq3$, iterate  Lemma \ref{lem:projection} down to a planar
independent environment of law $\nu_d$. Its avoidance sequence
is $b=f^{2d\downarrow4}$. For each nonempty
$A\subseteq\Dir_2$, let $m=|A|$. Then by \eqref{eq:gaps},
\begin{equation}\label{eq:localhit}
 \nu_d(S\cap A\neq\varnothing)
     =1-b_m\geq1-(9/20)^m.
\end{equation}
The right side is the hitting probability when each of the four
outgoing directions is independently included with probability
$11/20$. Note \eqref{eq:out-percolation-probability} for $p\in(0,1)$ and $d=2.$ By Theorem \ref{thm:local-global} and Lemma \ref{lem:one-arm},
\begin{eqnarray}\label{eq:comparison}
 \theta(\nu_d,2)\geq\theta_{\mathrm{arrow}}(11/20,2)
                         =\thetab(11/20,2).
\end{eqnarray}
Iterated application of \eqref{eq:thetaproj} proves
\eqref{eq:uniformbound}. Its right side is positive because
$11/20>1/2$.
\end{proof}

%

\renewcommand\thesection{\appendixname~\Alph{section}}
\setcounter{section}{0}
\section{\bf The local-to-global theorem and the one-arm identity}\label{app-1}
\setcounter{equation}{0}
\renewcommand\thetheorem{A.\arabic{theorem}}
\renewcommand\theequation{A.\arabic{equation}}
\noindent
Given any two probabilities $\mu$ and $\nu$ on $\mathcal{P}(\Dir_d)=\{A:\ A\subseteq \Dir_d\}$ such that for any $A\in \mathcal{P}(\Dir_d)$,
\begin{eqnarray*}
\sum\limits_{B\cap A\not=\varnothing}\mu(B)\leq \sum\limits_{B\cap A\not=\varnothing}\nu(B).
\end{eqnarray*}
For any $n\in\mathbb{Z}_+=\{0, 1, 2,\cdots\},$ let
$$B_n=\left\{x\in\mathbb{Z}^d:\ \Vert x\Vert_1\leq n\right\},\ \partial B_n=B_{n+1}\setminus B_n.$$
Recall Subsection \ref{sec:local}.

\begin{theorem}[Local-to-global \mbox{\cite[Theorem 2.1]{Bau-Ja2026}}]\label{thm:local-global} For any $n\in\mathbb{Z}_{+},$
$$\mathbb{P}_\mu^{+}[\mathbf{0}\rights \partial B_n]\leq \mathbb{P}_{\nu}^{+}[\mathbf{0}\rights \partial B_n].$$
\end{theorem}

For any $p\in [0,1],$ let $\mathbb{P}^{+}_p=\mathbb{P}_{\mu_p}^{+}$ with $\mu_p$ given by
$$\mu_p(A)=p^{\vert A\vert}(1-p)^{2d-\vert A\vert},\ A\subseteq\Dir_d;$$
and put
\begin{eqnarray}\label{eq:out-percolation-probability}
\theta_{\mathrm{arrow}}(p,d)=\theta_{\mathrm{arrow}}(\mu_p,d).
\end{eqnarray}
Write $\mathbb{P}_p$ for the undirected Bernoulli bond percolation measure on $\mathbb{Z}^d$ with density $p.$
Let $\mathbf{0}\leftrightsquigarrow\partial B_n$ be the event that $\partial B_n$ can be reached via an undirected open path from $\mathbf{0}$ in the model
$\mathbb{P}_p,$ and recall $\theta_{\rm b}(p,d)=\mathbb{P}_p[\mbox{The open cluster of}\ \mathbf{0}\ \mbox{is infinite}]$.

\begin{lemma}[One-arm identity \text{\cite[Lemma 3.5]{Bau-Ja2026}}]\label{lem:one-arm} For any $p\in (0,1),$
$$\mathbb{P}_p^{+}[\mathbf{0}\rights\partial B_n]=\mathbb{P}_p[\mathbf{0}\leftrightsquigarrow\partial B_n],\ n\in\mathbb{Z}_{+}.$$
Thus $\theta_{\mathrm{arrow}}(p,d)=\theta_{\rm b}(p,d).$
\end{lemma}

\bmhead*{Funding}
The project is supported partially by CNNSF (No.\,12671180) and by
Hu Xiang Gao Ceng Ci Ren Cai Ju Jiao Gong Cheng-Chuang Xin Ren Cai (No.\,2019RS1057).


\begin{thebibliography}{99}

\bibitem{Am-He2025}
G. Amir, M. Heydenreich, C. Hirsch. (2025). Planar reinforced $k$-out percolation. {\it Stoc. Process. Appl.} {\bf 189}, paper no. 104706, 11 pp.

\bibitem{Ba-Bo2013}
P. Balister, B. Bollob\'{a}s. (2013). Percolation in the $k$-nearest neighbor graph. In {\it Recent Results in Designs and Graphs: A Tribute to Lucia Gionfriddo, Quaderni di Matematica,} eds. M. Buratti, C. Lindner, F. Mazzocca and N. Melone.
Vol.\,{\bf 28}, pp.\,83-100.

\bibitem{Bau-Ja2026}
J.~B\"aumler, B.~Jahnel, J.~K\"oppl, B.~Lodewijks, L.~Reeves, A.~T\'obi\'as. (2026). Local criteria for global connectivity comparisons: beyond
stochastic domination (\url{https://arxiv.org/abs/2510.03934v2}). arXiv:2510.03934v2.


\bibitem{Co-He2026}
D.~Coupier, B.~Henry, B.~Jahnel, J.~K\"oppl. (2026). The planar lattice two-neighbor graph percolates (\url{https://arxiv.org/abs/2412.20781v1}). {\it Ann. Probab.} (to appear).

\bibitem{Ha-Me1996}
O. H\"{a}ggstr\"{o}m, R. Meester. (1996). Nearest neighbor and hard sphere models in continuum percolation. {\it Random. Structures. Algorithms.} {\bf 9}, 295-315.

\bibitem{Ha1960}
T. E. Harris. (1960). A lower bound for the critical probability in a certain percolation process. {\it Proc. Cambridge Philos. Soc.} {\bf 56}, 13-20.

\bibitem{Ja-Ko2025}
B.~Jahnel, J.~K\"oppl, B.~Lodewijks, A.~T\'obi\'as. (2025). Percolation in lattice $k$-neighbor graphs. {\it J. Appl. Probab.} {\bf 62(4)}, 1475-1492.

\bibitem{Ja-To2022}
B. Jahnel, A. T\'{o}bi\'{a}s. (2022). Absence of percolation in graphs based on stationary point processes with degrees bounded by two. {\it Random. Structures. Algorithms.} {\bf 62}, 240-255.

\bibitem{Ke1980}
H.~Kesten. (1980). The critical probability of bond percolation on the square lattice equals $1/2$. {\it Comm. Math. Phys.} {\bf 74(1)}, 41-59.
\end{thebibliography}
\end{document}